\documentclass[12pt]{amsart}
\usepackage{amsmath,amsthm,amsfonts,amssymb,mathrsfs}%\usepackage{paralist,nicefrac}
\usepackage[all]{xy}
\usepackage[total={6.5in,8.75in}, top=1.2in, left=1.0in, includefoot]{geometry}
\usepackage{amssymb}
\usepackage{latexsym}
\usepackage{amsmath,amsthm}
\usepackage{aeguill}
\usepackage{amssymb}
\usepackage{mathrsfs}
\usepackage{hyperref}
\usepackage{etoolbox}
\usepackage[all]{xy}
\usepackage{enumerate}
\usepackage[dvipsnames]{xcolor}
\usepackage{tikz}[2013/12/13]

\makeatletter
\pretocmd{\chapter}{\addtocontents{toc}{\protect\addvspace{15\p@}}}{}{}
\pretocmd{\section}{\addtocontents{toc}{\protect\addvspace{3\p@}}}{}{}
\makeatother

\makeatletter
\def\@tocline#1#2#3#4#5#6#7{\relax
  \ifnum #1>\c@tocdepth % then omit
  \else
    \par \addpenalty\@secpenalty\addvspace{#2}%
    \begingroup \hyphenpenalty\@M
    \@ifempty{#4}{%
      \@tempdima\csname r@tocindent\number#1\endcsname\relax
    }{%
      \@tempdima#4\relax
    }%
    \parindent\z@ \leftskip#3\relax \advance\leftskip\@tempdima\relax
    \rightskip\@pnumwidth plus4em \parfillskip-\@pnumwidth
    #5\leavevmode\hskip-\@tempdima
      \ifcase #1
       \or\or \hskip .5em \or \hskip 1em \else \hskip 1.5em \fi%
      #6\nobreak\relax
    \dotfill\hbox to\@pnumwidth{\@tocpagenum{#7}}\par
    \nobreak
    \endgroup
  \fi}
\makeatother
\newcommand{\C}{\mathbb{C}}

\newcommand{\Z}{\mathbb{Z}}

\newcommand{\Q}{\mathbb{Q}}
\newcommand{\F}{\mathbb{F}}
\newcommand{\A}{\mathbb{A}}

\newcommand{\cG}{\mathcal{G}}
\newcommand{\cH}{\mathcal{H}}

\newcommand{\Gal}{\operatorname{Gal}}

\renewcommand{\sc}{\operatorname{sc}}
\renewcommand{\ss}{\operatorname{ss}}
\newcommand{\der}{\operatorname{der}}

\newcommand{\End}{\operatorname{End}}

\newcommand{\SO}{\mathrm{SO}}

\newcommand{\GL}{\mathrm{GL}}

\newcommand{\SL}{\mathrm{SL}}

\newcommand{\Sp}{\mathrm{Sp}}
\newcommand{\GSp}{\mathrm{GSp}}
\newcommand{\GO}{\mathrm{GO}}

\def\bM{\mathbf{M}}
\def\bG{\mathbf{G}}
\def\bH{\mathbf{H}}

\def\bT{\mathbf{T}}

\newcommand\uM{\underline{M}}

\newcommand\uG{\underline{G}}

\newtheorem{thm}{Theorem}[section]

\newtheorem{cor}[thm]{Corollary}

\newtheorem{prop}[thm]{Proposition}

\newtheorem{defi}[thm]{Definition}

\begin{document}

\title[]{Monodromy of four dimensional irreducible compatible systems of $\Q$}

\author{Chun Yin Hui}
\email{chhui@maths.hku.hk, pslnfq@gmail.com}
\address{
Department of Mathematics\\
The University of Hong Kong\\
Pokfulam, Hong Kong
}

\maketitle

\begin{center}
\textit{In honor of Professor Michael Larsen's 60th birthday}
\end{center}

\begin{abstract}
Let $F$ be a totally real field and $n\leq 4$ a natural number.
We study the monodromy groups of any
$n$-dimensional strictly compatible system $\{\rho_\lambda\}_\lambda$
of $\lambda$-adic representations of $F$ with distinct Hodge-Tate numbers
such that $\rho_{\lambda_0}$ is irreducible for some $\lambda_0$.
When $F=\Q$, $n=4$, and $\rho_{\lambda_0}$ is fully 
symplectic, the following assertions are obtained. 
\begin{enumerate}[(i)]
\item  The representation $\rho_\lambda$ is fully symplectic for almost all $\lambda$.
\item If in addition the similitude character $\mu_{\lambda_0}$ of $\rho_{\lambda_0}$ is odd, then 
the system $\{\rho_\lambda\}_\lambda$ is potentially automorphic and the 
residual image $\bar\rho_\lambda(\Gal_\Q)$ has a subgroup conjugate to $\Sp_4(\F_\ell)$ for almost all $\lambda$.
\end{enumerate}
\end{abstract}

\section{Introduction}\label{s1}
\subsection{Main results}\label{mr}
Let $K$, $E$ be number fields, and $\{\rho_\lambda\}_\lambda$
an $E$-rational compatible system of semisimple $n$-dimensional 
$\lambda$-adic representations of $K$ \emph{in the sense of Serre} (Definition \ref{Sercs}).
Denote by $\Gamma_\lambda$ the \emph{monodromy} (i.e., image) of $\rho_\lambda$ 
and by $\bG_\lambda$ the \emph{algebraic monodromy group} of $\rho_\lambda$\footnote{The Zariski closure of the monodromy $\Gamma_\lambda$ in $\GL_{n,E_\lambda}$.}. 
A fundamental problem about compatible systems of $\lambda$-adic Galois representations
concerns whether the absolute irreducibility of $\rho_\lambda$ is independent of $\lambda$,
and furthermore, whether the residual representations $\bar\rho_\lambda$ are absolutely irreducible
for almost all $\lambda$ (i.e., all but finitely many $\lambda$) if some member 
of $\{\rho_\lambda\}_\lambda$ is absolutely irreducible,
to which establishing big images $\Gamma_\lambda$ 
 for all (or almost all) $\lambda$ is pivotal.
Two impetuses of the problem are the \emph{Mumford-Tate conjecture}
that concerns the $\lambda$-independence of $\bG_\lambda^\circ$ for motivic compatible systems
and the \emph{irreducibility conjecture}  
that predicts the absolute irreducibility of $\rho_\lambda$
for compatible systems attached to algebraic cuspidal automorphic forms of $\GL_n(\A_K)$.
There have been many studies on the irreducibility and big images of 
compatible system since 70s, 
assuming that $\{\rho_\lambda\}_\lambda$ is 
\begin{itemize} 
\item motivic, see \cite{Se72,Se85},\cite{Ri76} for abelian varieties, \cite{DV04},\cite{DW21} for $n=3$, and \cite{DV08,DV11} for $n=4$,
\item automorphic, see \cite{Ri77,Ri85},\cite{Ta95},\cite{Dim05} for (Hilbert) modular forms, \cite{BR92},\cite{DV04} for $n=3$, \cite{Die02b,Die07},\cite{Ra13},\cite{DZ20},\cite{We22} for $n=4$, \cite{CG13},\cite{Xi19},\cite{Hu22} for $n\leq 6$, and \cite{BLGGT14} in general,
\end{itemize}
and more recently, 
\begin{itemize}
\item a weakly compatible system (Definition \ref{scs}), see 
\cite{BLGGT14},\cite{PT15},\cite{LY16},\cite{PSW18}, and \cite{DW21}.
\end{itemize}

\noindent In particular in the recent work \cite{Hu22},
when $K=F$ is totally real or CM, $n\leq 6$, and
\begin{equation}\label{sc1}
\{\rho_\lambda:\Gal_F\to\GL_n(\overline{E}_\lambda)\}_\lambda
\end{equation}
is the strictly compatible system (Definition \ref{scs})
attached to a regular algebraic, polarized, cuspidal automorphic 
representation of $\GL_n(\A_F)$, we proved (\cite[Theorem 1.4]{Hu22}) that 
\begin{enumerate}[(I)]
\item $\rho_\lambda$ is irreducible  for almost all $\lambda$ and
\item $\rho_\lambda$ is residually irreducible for almost all $\lambda$ if in addition $F=\Q$
\end{enumerate} 
by combining various Galois theoretic and potential automorphy results.
The innovation of \cite{Hu22} is the development of some big images results for 
subrepresentations of the compatible system $\{\rho_\lambda\}_\lambda$ 
inspired by previous works \cite{Hu15} and \cite{HL16,HL20} 
(see Theorem \ref{general} and Proposition \ref{use1}).
The strategy of the irreducibility result (I), roughly speaking, is that if there exists an infinite set $\mathcal L$ such that $\rho_\lambda$ is reducible whenever $\lambda\in\mathcal L$,
then some low dimensional subrepresentations of $\rho_\lambda$ 
(or representation constructed from $\rho_\lambda$)
for some $\lambda\in\mathcal L$ have big images
so that the potential automorphy theorems in \cite{BLGGT14} can be applied
 to draw a contradiction to the irreducibility of a member of $\{\rho_\lambda\}_\lambda$ 
due to Patrikis-Taylor \cite{PT15}. The residual irreducibility result (II) 
requires an extra input from Serre's modularity conjecture \cite{Se87} 
(proven in \cite{KW09a,KW09b}).
In this article, we use these ideas to study the monodromy
of some low dimensional strictly compatible systems of totally real field $F$.
Given an algebraic extension $E_\ell$ of $\Q_\ell$, 
the algebraic monodromy group $\bG_\ell$ of an $\ell$-adic Galois representation $\sigma_\ell:\Gal_K\to\GL_n(E_\ell)$ is the $E_\ell$-subgroup of $\GL_{n,E_\ell}$ defined as the Zariski closure of
the monodromy $\sigma_\ell(\Gal_K)$ in $\GL_{n,E_\ell}$.

\begin{thm}\label{mt1}
Suppose $F$ is a totally real field, $n\leq 4$, and  $\{\rho_\lambda:\Gal_F\to\GL_n(\overline{E}_\lambda)\}_\lambda$ is a strictly compatible system of $F$ with distinct $\tau$-Hodge-Tate numbers for each $\tau:F\to\overline{E}_\lambda$. Let $\bG_\lambda$ be the algebraic monodromy of $\rho_\lambda$.
If $\rho_{\lambda_0}$ 
is irreducible for some $\lambda_0$, then the following assertions hold.
\begin{enumerate}[(i)]
\item When $n\leq 3$, the representation $\rho_\lambda$ is irreducible for all $\lambda$.
\item When $n\leq 3$, the identity component $\bG_\lambda^\circ\subset\GL_{n,\overline{E}_\lambda}$ is independent of $\lambda$.
\item When $n\leq 3$, the representation $\rho_\lambda$ is residually irreducible for almost all $\lambda$. 
\item When $n=4$, the representation $\rho_\lambda$ is irreducible for almost all  $\lambda$. 
\end{enumerate}
\end{thm}

When $F=\Q$ and $n=4$, we have a better description of the monodromy by Serre's modularity conjecture. If $\bG_\lambda\subset\GO_n$ (resp. $\GSp_n$) with respect to a non-degenerate symmetric (resp. skew-symmetric) 
 pairing $\left\langle~,~\right\rangle$ on the representation space of $\rho_\lambda$, 
there is a similitude character $\mu_\lambda$ such that $\rho_\lambda\cong \rho_\lambda^\vee\otimes\mu_\lambda$.
The character $\mu_\lambda$ is said to be \emph{odd} (resp. \emph{even}) if $\mu_\lambda(c)=-1$ (resp. $1$)
where $c\in\Gal_\Q$ is a complex multiplication.
Denote by $\bG_\lambda^{\der}$ the derived group of the identity component $\bG_\lambda^\circ$.
The representation $\rho_\lambda$ is said to be
\begin{itemize}
\item \emph{fully orthogonal} if $\bG_\lambda^{\der}=\SO_n$ with $\left\langle~,~\right\rangle$ symmetric;
\item \emph{fully symplectic}  if $\bG_\lambda^{\der}=\Sp_n$ with $\left\langle~,~\right\rangle$ skew-symmetric.
\end{itemize} 
Below is our main result on four dimensional fully symplectic Galois representations 
of $\Q$; see \cite{LY16} for four dimensional fully orthogonal Galois representations.

\begin{thm}\label{mt2}
Suppose $\{\rho_\lambda:\Gal_\Q\to\GL_4(\overline{E}_\lambda)\}_\lambda$ is a strictly compatible system of $\Q$ with distinct Hodge-Tate numbers. If $\rho_{\lambda_0}$ is fully symplectic\footnote{With respect to a non-degenerate skew-symmetric pairing on the representation space of $\rho_{\lambda_0}$.} for some $\lambda_0$, then 
the following assertions hold.
\begin{enumerate}[(i)]
\item The representation $\rho_\lambda$ is fully symplectic for almost all $\lambda$.
\item If in addition the similitude character $\mu_{\lambda_0}$ is odd, then $\{\rho_\lambda\}_\lambda$ is potentially automorphic and 
the residual image $\bar\rho_\lambda(\Gal_\Q)$ has a subgroup conjugate to $\Sp_4(\F_\ell)$ for almost all $\lambda$.
\end{enumerate}
\end{thm}

When the strictly compatible system $\{\rho_\lambda\}_\lambda$ is also a Serre compatible system defined over $E=\Q$, the prime $\lambda$ is a rational prime $\ell$, the algebraic monodromy group $\bG_\ell$ is defined over $\Q_\ell$, and the monodromy group $\Gamma_\ell$ is a compact open subgroup of $\bG_\ell(\Q_\ell)$. In \cite{Lar95}, Larsen conjectured 
that for any motivic compatible system
the subgroup $\Gamma_\ell^{\sc}\subset\bG_\ell^{\sc}(\Q_\ell)$ (Definition \ref{gscd}) 
is \emph{hyperspecial maximal compact} (see \cite{Ti79})
when $\ell$ is sufficiently large. We obtain the corollary below.

\begin{cor}\label{mc}
Suppose $\{\rho_\ell:\Gal_\Q\to\GL_4(\Q_\ell)\}_\ell$ is a strictly compatible system of $\Q$
with distinct Hodge-Tate numbers.
If $\rho_{\ell_0}$ is fully symplectic with odd similitude character 
$\mu_{\ell_0}$ for some $\ell_0$, 
then $\Gamma_\ell^{\sc}$ is a hyperspecial maximal compact subgroup of 
$\bG_\ell^{\sc}(\Q_\ell)=\bG_\ell^{\der}(\Q_\ell)\cong\Sp_4(\Q_\ell)$ (that is, $\Gamma_\ell^{\sc}\cong\Sp_4(\Z_\ell)$)
for all $\ell\gg0$.
\end{cor}

\subsection{Remarks and comparisons of techniques}
An $\ell$-adic representation $\sigma_\ell:\Gal_K\to \GL_m(\overline\Q_\ell)$
is said to be \emph{of type A} if every simple factor of the Lie algebra $\mathfrak{g}_\ell$ of 
the algebraic monodromy group of the semisimplification $\sigma_\ell^{\ss}$
is of type A in the Killing-Cartan classification, i.e.,
$$[\mathfrak{g}_\ell,\mathfrak{g}_\ell]\cong \bigoplus_i \mathfrak{sl}_{k_i,\overline\Q_\ell}.$$
Note that $\sigma_\ell$ is of type A when $\dim\sigma_\ell=m\leq 3$.

For a semisimple Serre compatible system $\{\rho_\lambda\}_\lambda$ satisfying 
reasonable local conditions (Theorem \ref{general}(a),(b)), 
the big images results in \cite{Hu22} give lower bounds for the residual images
of subrepresentations $\sigma_\lambda$ in terms of the \emph{formal bi-characters} 
(Definition \ref{fc}) of the algebraic monodromy of $\sigma_\lambda$ 
for almost all $\lambda$ (Theorem \ref{general}(i)-(iii)). 
This result works particularly well if $\sigma_\lambda$ 
is of type A (Theorem \ref{general}(iv)-(v))
and can be used to streamline some arguments or improve some results in other works.

\subsubsection{Case $n=2$} Theorem \ref{mt1}(i) (and $F$ is any number field) is essentially due to Ribet \cite{Ri77} (see also \cite[Proposition 2.3.1]{BR92}) by using only the fact that if $\rho_{\lambda_1}$ is reducible for some $\lambda_1$ then it is locally algebraic,
which implies $\{\rho_\lambda\}_\lambda$ is a sum of two one-dimensional systems.
This contradicts that $\rho_{\lambda_0}$ is irreducible.

Theorem \ref{mt1}(iii) for Serre compatible system $\{\rho_\lambda\}_\lambda$ attached to objects like non-CM elliptic curves \cite{Se72} and (Hilbert) modular forms \cite{Ri85},\cite{Dim05} 
is obtained by analyzing some explicit data of the object, e.g., sizes of Fourier coefficients.
Since $\{\rho_\lambda\}_\lambda$ is of type A and is either motivic or strictly compatible, the results
follow directly from our big images results (Theorem \ref{general}(v)).

\subsubsection{Case $n=3$} 
In \cite{DW21}, it is proved that some three dimensional, regular self-dual compatible system $\{\rho_\ell\}_\ell$ arising from a surface (an elliptic fibration over the projective line) over $\Q$ is either (a) absolutely irreducible for a Dirichlet density one set of rational primes $\ell$ or (b) $\{\rho_\ell\}_\ell$ decomposes as a sum of two irreducible compatible systems. The density one set of primes condition appears 
in case (a) because some results of \cite{CG13} and \cite{BLGGT14} (that hold under density one\footnote{This is because works like \cite{BLGGT14} rely on a (Dirichlet density one) big images result of Larsen \cite{Lar95} for Serre compatible systems.}) are used. Case (a) can be improved to be absolutely irreducible for all $\ell$ (Theorem \ref{mt1}(i)).

In \cite{DV04}, it is proved that some three dimensional non-self-dual motivic compatible system
 $\{\rho_\lambda\}_\lambda$ of $\Q$ have big images 
(in particular $\bar\rho_\lambda$ is absolutely irreducible) if $\ell:=\text{char}(\lambda)$ 
belongs to a Dirichlet density one set of rational primes. 
The techniques (also in other works of Dieulefait \cite{Die02b},\cite{DV08,DV11},\cite{DZ20}) include studying the tame inertia characters (\cite{Se72}) for large $\ell$
by relating them to Hodge-Tate weights (\cite{FL82})
and the classification of maximal subgroups
of $\uG(\F_q)$ where $\uG$ is some connected $\F_q$-semisimple group.
Since $\{\rho_\lambda\}_\lambda$ is of type A and motivic, it follows from 
Theorem \ref{general}(v) that
residual (absolute) irreducibility holds
for almost all $\lambda$.

\subsubsection{Case $n=4$} Irreducibility and big images for four dimensional compatible system $\{\rho_\lambda\}_\lambda$ of $\Q$ are studied in the works \cite{Die02b},\cite{DV11},\cite{Ra13},\cite{LY16},\cite{DZ20},\cite{We22}, in which (potential) automorphy techniques 
such as Serre's conjecture and \cite{BLGGT14} are used to rule out two dimensional factors of $\rho_\lambda$. This is also the main idea for Theorem \ref{mt1}(iv) and Theorem \ref{mt2}. 
In order to apply such automorphy theorems, an oddness condition and 
a big image condition
are usually required (see Theorem \ref{pauto}(2),(4)). The first condition for
some two dimensional representation of a totally real field $F$ is
satisfied by a result of Calegari \cite{Ca11} (see Proposition \ref{odd}).
The second one can be obtained by Theorem \ref{general} if the formal character of $\bG_\lambda^{\der}$ is known. For example, if $\bG_{\lambda_0}^{\der}=\Sp_4$ then $\bG_\lambda^{\der}$ (in $\GL_4$) has only three possibilities: $\Sp_4$, $\SO_4$, $\SL_2\times\SL_2$ (by Theorem \ref{hui1}).
Big images results work well in the last two cases because they are of type A.
If the strict compatibility of $\{\rho_\lambda\}_\lambda$ and the
fully symplectic (for some $\lambda_0$) condition are replaced by weak compatibility and purity, 
then the potential automorphy assertion in Theorem \ref{mt2}(ii) 
is obtained in \cite[Theorem A]{PT15}.

\subsection{Structure of the article}
Section $2$ presents the preliminaries for studying monodromy groups of compatible systems,
including some conjectures on compatible systems 
of Galois representations and some essential results for later use, 
for example, a potential automorphy theorem \cite[Theorem C]{BLGGT14}
and some big images results in \cite{Hu22} (summarized as Theorem \ref{general}). 
Section $3$ is devoted to the proofs of the main results in $\mathsection\ref{mr}$.

\section*{Acknowledgments}
After the collaboration \cite{HL20}, 
Professor Larsen asked if there are some general four dimensional compatible systems $\{\rho_\ell\}_\ell$ with $\GSp_4$ 
algebraic monodromy that fulfill the big image criterion $(\ast)$ in Theorem \ref{inv},
i.e., $\rho_\ell$ has to be residually (absolutely) irreducible for all $\ell\gg0$. 
Larsen's question inspired the work \cite{Hu22}, which is used in this article essentially
to obtain Theorem \ref{mt2} and Corollary \ref{mc}, that to some extent, get back to his question.
My debt to the work of Professor Larsen on Galois representations is obvious 
and I am grateful to him for his collaboration and inspiration.
I would like to thank Professors Bo-Hae Im, Ravi Ramakrishna, and Pham Huu Tiep 
for organizing the conference ``Algebra 2022 and beyond'' in honor of Professor Larsen's 60th birthday 
and their kind invitation.

I would like to thank Haining Wang for his interest in the article.
I would like to thank the referee for helpful comments and suggestions
on the exposition, content, and references of the article. 
The work described in this article was partially supported by a grant 
from the Research Grants Council of the Hong Kong Special Administrative
Region, China (Project No. 17302321).

\section{Preliminaries}
\subsection{Compatible systems} Let $K$ and $E$ be number fields and denote their sets of finite
places by $\Sigma_K$ and $\Sigma_E$ respectively. For any prime number $\ell$,
denote by $S_\ell$ the set of elements of $\Sigma_K$ dividing $\ell$.
For $\lambda\in\Sigma_E$, denote by $\ell$ the residue characteristic of $\lambda$.
Let \begin{equation}\label{sc2}
\{\rho_\lambda:\Gal_K\to \GL_n(\overline{E}_\lambda)\}_{\lambda\in\Sigma_E}
\end{equation}
be a family of $\lambda$-adic representations of $K$.

\begin{defi}\label{scs}
The family \eqref {sc2} is called a weakly compatible system defined over $E$ (in the sense of \cite[$\mathsection5.1$]{BLGGT14}) if $\rho_\lambda$ is semisimple for all $\lambda$ and 
there exist a finite subset $S\subset\Sigma_K$ and a polynomial $P_v(t)\in E[t]$ for each $v\in\Sigma_K\backslash S$ such that
the following conditions (a),(b), and (c) hold.
\begin{enumerate}[(a)]
\item For each $\lambda\in\Sigma_E$  (with residue characteristic $\ell$)
and $v\in\Sigma_K\backslash (S\cup S_\ell)$, 
the representation $\rho_\lambda$ is unramified at $v$ and 
the characteristic polynomial of the image $\rho_\lambda(Frob_v)$ of the Frobenius class at $v$ is $P_v(t)$.
\item Each $\rho_\lambda$ is de Rham at all places $v\in S_\ell$ and
moreover crystalline if $v\notin S$.
\item For each embedding $\tau: K\hookrightarrow \overline E$, the set 
of $\tau$-Hodge-Tate numbers of $\rho_{\lambda}$ is independent of $\lambda$ and any $\overline E\hookrightarrow \overline E_\lambda$ over $E$.
\end{enumerate}
A weakly compatible system is called a strictly compatible system if condition (d) below holds.
\begin{enumerate}[(d)]
\item If $v\in\Sigma_K$ does not divide the residue characteristic of $\lambda$, 
the semisimplified Weil-Deligne representation 
$\iota \text{WD}(\rho_{\lambda}|_{\Gal_{K_v}})^{F-ss}$ 
is independent of $\lambda$ and $\iota:\overline{E}_\lambda\stackrel{\cong}{\rightarrow} \C$.
\end{enumerate}
\end{defi}

\begin{defi}\label{Sercs}(see \cite[Chapter 1]{Se98})
The family \eqref{sc2} is called a Serre compatible system of $K$ defined over $E$
if $\rho_\lambda(\Gal_K)\subset\GL_n(E_\lambda)$ for all $\lambda$
and the family satisfies the compatibility condition \ref{scs}(a) 
for some finite $S\subset\Sigma_K$ and $P_v(t)\in E[t]$ for each $v\in\Sigma_K\backslash S$.
\end{defi}

Under the distinct $\tau$-Hodge-Tate numbers condition, a strictly compatible system \eqref{sc2}
can be identified as a Serre compatible system after enlarging $E$.

\begin{prop}\label{use} \cite[Lemma 5.3.1(3)]{BLGGT14}
Let $\{\rho_\lambda:\Gal_K\to\GL_n(\overline{E}_\lambda)\}_\lambda$ be a strictly compatible system of $K$ defined over $E$ with distinct
$\tau$-Hodge-Tate numbers. After enlarging $E$, we may suppose $\{\rho_\lambda\}_\lambda$
is a semisimple Serre compatible system defined over $E$.
\end{prop}

\subsection{Automorphy of Galois representations} When $K$ is totally real or CM, one can attach an $n$-dimensional strictly compatible system of $K$
defined over some CM field $E$ to a regular algebraic, polarized, cuspidal automorphic representation $\pi$
of $\GL_n(\A_K)$ (see \cite[$\mathsection2.1$]{BLGGT14}) such that the L-functions agree. 
Conversely, the Fontaine-Mazur-Langlands conjecture \cite{FM95},\cite{Lan79} (see also \cite{Cl90},\cite{Ta04})
asserts that any irreducible $\ell$-adic representation $\rho_\ell: \Gal_\Q\to\GL_n(\overline\Q_\ell)$
 which is unramified at all but finitely many primes and with $\rho_\ell|_{\Gal_{\Q_\ell}}$ de Rham
comes from a cuspidal automorphic representation of $\GL_n(\A_\Q)$. For $n=2$, many cases of this
conjecture  are established by Kisin \cite{Ki09} and Emerton \cite{Em11} when $\rho_\ell$ is residually irreducible
and recently by Pan \cite{Pa21} when $\rho_\ell$ is residually reducible. For general $n$, we have the 
following potential automorphy theorem for totally real field. Denote $\zeta_\ell:=e^{2\pi i/\ell}$.

\begin{thm}\cite[Theorem C]{BLGGT14}\label{pauto}
Suppose $F$ is a totally real field. Let $m$ be a positive integer and let $\ell\geq 2(m+1)$ be a prime. Let
$$\rho_\ell:\Gal_F\to\GL_m(\overline\Q_\ell)$$
be a continuous representation. Suppose the following conditions are satisfied.
\begin{enumerate}[(1)]
\item (Unramified almost everywhere): $\rho_\ell$ is unramified at all but finitely many primes.
\item (Odd essential self-duality): either $\rho_\ell$ maps to $\GSp_m$ with totally odd similitude character
or it maps to $\GO_m$ with totally even similitude character.
\item (Potential diagonalizability and regularity): $\rho_\ell$ is potentially diagonalizable (and hence potentially crystalline)
at each prime $v$ of $F$ above $\ell$ and for each $\tau: F\hookrightarrow \overline\Q_\ell$ 
it has $m$-distinct $\tau$-Hodge-Tate numbers.
\item (Irreducibility): $\rho_\ell|_{\Gal_{F(\zeta_\ell)}}$ is residually irreducible.
\end{enumerate}
Then we can find a finite Galois totally real extension $F'/F$ such that $\rho_\ell|_{\Gal_{F'}}$
is automorphic (i.e., attached to a regular algebraic, (polarized) cuspidal automorphic representation of $\GL_m(\A_{F'})$). 
Moreover, $\rho_\ell$ is part of a strictly (pure) compatible system of $F$. 
\end{thm}

Given a strictly (and also Serre) compatible system 
$\{\rho_\lambda:\Gal_F\to\GL_n(E_\lambda)\}_\lambda$ of a totally real field $F$, in order to apply Theorem \ref{pauto}
to a subrepresentation $r_\lambda$ of $\rho_\lambda$ we need $r_\lambda$ to satisfy the four conditions.
The first condition is automatic and the third one holds when $\ell$ is sufficiently large 
(by conditions \ref{scs}(b),(c) and \cite[Lemma 1.4.3(2)]{BLGGT14}). 
For the second one (odd essential self-duality), 
we have the following result when $\dim r_\lambda=2$.

\begin{prop}\cite[Proposition 2.5]{CG13}\label{odd}
Suppose $r_\ell:\Gal_F\to\GL_2(\overline\Q_\ell)$ is a continuous
representation of a totally real field $F$ and  $\ell>7$. Assume that
\begin{enumerate}[(1)]
\item $r_\ell$ is unramified outside of finitely many primes.
\item $\mathrm{Sym}^2 r_\ell|_{\Gal_{F(\zeta_\ell)}}$ is residually irreducible.
\item $\ell$ is unramified in $F$.
\item For each place $v|\ell$ of $F$ and each $\tau:F_v\hookrightarrow \overline\Q_\ell$,
the $\tau$-Hodge-Tate numbers of $r_\ell|_{\Gal_{F_v}}$ is a set of two distinct integers
whose difference is less than $(\ell-1)/2$, and $r_\ell|_{\Gal_{F_v}}$ is crystalline.
\end{enumerate}
Then, the pair $(r_\ell,\det r_\ell)$ is essentially self-dual and odd.
\end{prop}

Hence, for $\dim r_\lambda=2$ and $\ell\gg0$ the conditions (1),(3),(4) of 
Proposition \ref{odd} hold (again using \ref{scs}(b),(c)).
Both conditions Theorem \ref{pauto}(4) and Proposition \ref{odd}(2)
require $r_\lambda$ to have big image, which 
are fulfilled if the algebraic monodromy group of $r_\lambda$ contains $\SL_2$
and $\ell\gg0$ by applying Theorem \ref{general}(iv) and Proposition \ref{use1} to the compatible system $\{\rho_\lambda\}_\lambda$. 

\subsection{Big Galois images and irreducibility}\label{bgi}
Our works on big Galois images \cite{HL16},\cite{HL20}, and \cite{Hu22} are motivated by 
a well-known theorem of Serre on Galois actions of non-CM elliptic curves \cite{Se72} and also
a conjecture of Larsen \cite{Lar95} on Galois actions of motivic compatible systems (see also \cite{Se94}). 
Let $X$ be a smooth projective variety defined over a number field $K$
and $w\geq0$ an integer. The \'etale cohomology group $V_\ell:=H^w(X_{\overline K},\Q_\ell)$
is acted on by $\Gal_K$ for all primes $\ell$. By \cite{De74}, the family of $\ell$-adic representations
\begin{equation}\label{sc3}
\{\rho_\ell:\Gal_K\to\GL(V_\ell)\cong\GL_n(\Q_\ell)\}_{\ell\in\Sigma_\Q}
\end{equation}
is a Serre compatible system of $K$ defined over $\Q$ (Definition \ref{Sercs}).  
Let $\Gamma_\ell$ (resp. $\bG_\ell$) be the image (resp. algebraic monodromy group) of $\rho_\ell$.
Consider the diagram
\begin{equation}\label{gsc}
\xymatrix{
     &  \bG_\ell^{\sc}(\Q_\ell)\ar[d]^{\pi_\ell}\\
\bG_\ell^\circ(\Q_\ell) \ar[r]^{q_\ell} &\bG_\ell^{\ss}(\Q_\ell)}
\end{equation}
where $\bG_\ell^{\ss}$ is the quotient of the connected linear algebraic group $\bG_\ell^{\circ}$ by radical,
$\bG_\ell^{\sc}$ is the universal covering of $\bG_\ell^{\ss}$, and $q_\ell$ (resp. $\pi_\ell$)
is given by the quotient (resp. covering) morphism. 

\begin{defi}\label{gscd}
Denote by $\Gamma_\ell^{\sc}$ the compact subgroup $\pi_\ell^{-1}(q_\ell(\Gamma_\ell\cap\bG_\ell^\circ(\Q_\ell)))$
of $\bG_\ell^{\sc}(\Q_\ell)$ in diagram \eqref{gsc}.
\end{defi}

Larsen conjectured \cite{Lar95} that for all $\ell\gg0$, the subgroup $\Gamma_\ell^{\sc}\subset \bG_\ell^{\sc}(\Q_\ell)$ is
 \emph{hyperspecial maximal compact}, i.e., 
there exists a semisimple group scheme $\cG_\ell$ defined over $\Z_\ell$ with generic fiber $\bG_\ell^{\sc}$ 
such that $\cG_\ell(\Z_\ell)=\Gamma_\ell^{\sc}$. He proved that for a set of
primes $\ell$ of Dirichlet density one the assertion is true \cite[Theorem 3.17]{Lar95}.
We proved the conjecture when (the semisimplification of) 
$\rho_\ell$ is of type A for $\ell\gg0$ \cite{HL16} and 
when $X$ is an abelian variety or hyper-K\"ahler variety (and $w=2$) \cite[Theorem 1.3]{HL20}.
We also established the following criterion.

\begin{thm}\label{inv}\cite[Theorem 1.2]{HL20}
Let $\{\rho_\ell\}_\ell$ be the Serre compatible system  arising from the $w$th cohomology 
of a smooth projective variety $X/K$ such that $\bG_\ell$ is 
connected for all $\ell$.
For all sufficiently large $\ell$, the subgroup $\Gamma_\ell^{\sc}\subset \bG_\ell^{\sc}(\Q_\ell)$ 
is hyperspecial maximal compact if and only if 
\begin{equation*}
\tag{$\ast$}
\dim_{\Q_\ell} \End_{\Gal_K}(V_\ell)=\dim_{\F_\ell}\End_{\Gal_K}(\overline V_\ell^{\ss}),
\end{equation*}
where $\overline V_\ell^{\ss}$ is the semisimplified reduction of $V_\ell$.
\end{thm}

Not only is $\Gamma_\ell\subset\bG_\ell(\Q_\ell)$ conjectured to be large for $\ell\gg0$,
but the (general) Mumford-Tate conjecture \cite[$\mathsection9$]{Se94} asserts the $\ell$-independence of 
algebraic monodromy group $\bG_\ell$ in the sense that if $\bG_{MT}$ 
is the Mumford-Tate group of $H^w(X(\C),\Q)$, then
for all $\ell$ we have
$$\bG_\ell^\circ\cong \bG_{MT}\times_\Q\Q_\ell.$$ 
For semisimple Serre compatible systems, we obtained the following 
$\lambda$-independence result (Theorem \ref{hui1}) on 
the algebraic monodromy groups. First we give a definition.

\begin{defi}\label{fc}
Let $F$ be a field and  $\bH\subset\GL_{n,F}$ a reductive subgroup.
\begin{enumerate}[(a)]
\item The formal character 
of $\bH$ is a subtorus $\bT$ in $\GL_{n,\overline{F}}$ up to conjugation such that
$\bT$ is a maximal torus of $\bH\times\overline{F}$. 
\item The formal bi-character of $\bH$
is a chain of subtori $\bT'\subset\bT$ in $\GL_{n,\overline{F}}$ up to conjugation such that
$\bT$ is a maximal torus of $\bH\times\overline{F}$ and $\bT'$ is a maximal torus 
of $\bH^{\der}\times\overline{F}$, where $\bH^{\der}$ denotes the derived group of the identity component $\bH^\circ$.
\end{enumerate}
\end{defi}

\begin{thm}\label{hui1}\cite[Theorem 3.19, Remark 3.22]{Hu13}
Let $\{\phi_\lambda\}_\lambda$ be a semisimple $n$-dimensional Serre compatible system of $K$ (defined over $E$). 
Then the formal bi-character of the algebraic monodromy group $\bH_\lambda\subset\GL_{n,E_\lambda}$ is independent 
of $\lambda$\footnote{It means that the formal bi-characters have
the same $\Z$-form in $\GL_{n,\Z}$.}. In particular, the rank (resp. semisimple rank)
of $\bH_\lambda$ is independent of $\lambda$.
\end{thm}

Moreover, we extended  in \cite{Hu15} this $\lambda$-independence result to the mod $\ell$ 
motivic compatible system $\{\bar\rho_\ell^{\ss}:\Gal_K\to\GL(\overline V_\ell^{\ss})\}_\ell$. 
More precisely, for $\ell\gg0$ we constructed a connected reductive subgroup $\uG_\ell\subset\GL_{n,\F_\ell}$ 
called the \emph{algebraic envelope}\footnote{Serre constructed algebraic envelopes to study
 Galois actions on $\ell$-torsions of abelian varieties for $\ell\gg0$ \cite{Se86}.} of $\rho_\ell$
 such that $\uG_\ell(\F_\ell)$ is uniformly commensurate with $\bar\Gamma_\ell$ (the image of $\bar\rho_\ell^{\ss}$)
for all $\ell\gg0$ and the formal bi-character of $\bG_\ell$ and $\uG_\ell$ coincide for all $\ell\gg0$ \cite[Theorem A]{Hu15}. The works \cite{HL16},\cite{HL20}, and \cite{Hu22}
are all based on \cite{Hu15}.
Inspired by some group theoretic techniques developed in \cite{HL20}, 
we constructed \emph{algebraic envelopes of subrepresentations} (see \cite[$\mathsection3$]{Hu22})
of Serre compatible systems satisfying certain local conditions (originated from motivic
compatible system \eqref{sc3}) to prove the following big images results.
Denote by $\bar\epsilon_\ell$ the mod $\ell$ cyclotomic character.

\begin{thm}\label{general}\cite[Theorems 3.12(i),(ii),(iii),(v), and Theorem 1.2]{Hu22}
Let $\{\rho_\lambda\}_{\lambda}$ be a semisimple Serre compatible system of $K$ defined over $E$. 
Suppose there exist some integers $N_1,N_2\geq 0$ and finite extension $K'/K$ such that the following conditions hold.
\begin{enumerate}[(a)]
\item (Bounded tame inertia weights): for almost all $\lambda$ 
and each finite place $v$ of $K$ above $\ell$, 
the tame inertia weights of the local representation 
$(\bar\rho_{\lambda}^{\ss}\otimes\bar\epsilon_\ell^{N_1})|_{\Gal_{K_v}}$ belong to $[0,N_2]$.
\item (Potential semistability): for almost all $\lambda$ and each finite place $w$ of $K'$ not above $\ell$,
the semisimplification of the local representation $\bar\rho_{\lambda}^{\ss}|_{\Gal_{K_{w}'}}$ is unramified.
\end{enumerate}
Then there exists a finite Galois extension $L/K$ such that for almost all $\lambda$
and for each subrepresentation $\sigma_\lambda:\Gal_K\to \GL(W_\lambda)$ of $\rho_\lambda\otimes\overline\Q_\ell$,
the following assertions hold.
\begin{enumerate}[(i)]
\item There is a connected reductive subgroup $\uG_{W_{\lambda}}\subset\GL_{\overline W_\lambda^{\ss}}$ (called the algebraic envelope of $\sigma_\lambda$)
that is  semisimple on $\overline W_\lambda^{\ss}$ (semisimplified reduction of $W_\lambda$) and contains $\bar\sigma_\lambda^{\ss}(\Gal_L)$.
\item The commutants of $\bar\sigma_\lambda^{\ss}(\Gal_L)$ and $\uG_{W_{\lambda}}$ 
(resp. $[\bar\sigma_\lambda^{\ss}(\Gal_L),\bar\sigma_\lambda^{\ss}(\Gal_L)]$ and $\uG_{W_{\lambda}}^{\der}$) in $\End(\overline W_\lambda^{\ss})$ are equal.
In particular, $\bar\sigma_\lambda^{\ss}(\Gal_L)$ (resp. $[\bar\sigma_\lambda^{\ss}(\Gal_L),\bar\sigma_\lambda^{\ss}(\Gal_L)]$) is irreducible on $\overline W_{\lambda}^{\ss}$ if and only if 
$\uG_{W_{\lambda}}$ (resp. $\uG_{W_{\lambda}}^{\der}$) is irreducible on $\overline W_{\lambda}^{\ss}$.
\item The algebraic envelope $\uG_{W_{\lambda}}$ and algebraic monodromy $\bG_{W_\lambda}$ of  $\sigma_\lambda$
 have the same formal bi-character. Moreover, this formal bi-character has
finitely many possibilities depending on $\{\rho_\lambda\}_{\lambda}$.
\item If $\sigma_\lambda$ is Lie-irreducible\footnote{It means $\sigma_\lambda|_{\Gal_{M}}$
is irreducible for all finite extensions $M/K$, or equivalently, 
the identity component $\bG_{W_\lambda}^\circ$ is irreducible on $W_\lambda$.} and of type A (see  $\mathsection1.2$), 
then $\uG_{W_{\lambda}}$ 
and $\Gal_{K^{ab}}$ are irreducible on $\overline W_{\lambda}^{\ss}$,
where $K^{ab}/K$ is the maximal abelian extension.
\item If $\sigma_\lambda$ is irreducible and of type A, then it is residually irreducible.
\end{enumerate}
\end{thm}

\subsection{Potential automorphy of low dimensional subrepresentations}
The following results will be useful in next section.

\begin{prop}\label{use1} 
Let $\{\rho_\lambda:\Gal_K\to\GL_n(E_\lambda)\}_\lambda$ be a strictly (and also Serre) compatible system of $K$ defined over $E$.
Then it satisfies the conditions (a) and (b) of Theorem \ref{general}.
\end{prop}

\begin{proof}
Verification of conditions (a) and (b) are the same as \cite[Theorem 4.1]{Hu22} using the strict compatibility conditions.\end{proof}

\begin{prop}\label{use2}
Let $\{\rho_\lambda:\Gal_F\to\GL_n(E_\lambda)\}_\lambda$ be a strictly (and also Serre) 
compatible system of a totally real field $F$ defined over $E$. For almost all $\lambda$, 
if $\sigma_\lambda:\Gal_F\to\GL(W_\lambda)$ is an $m$-dimensional subrepresentation
of $\rho_\lambda\otimes\overline{E}_\lambda$ such that the 
$\tau$-Hodge-Tate numbers of $\sigma_\lambda$ are distinct and the algebraic monodromy group $\bG_{W_\lambda}$ is one of the following cases:
\begin{enumerate}[(a)]
\item $m=2$ and $\bG_{W_\lambda}^{\der}=\SL_2$;
\item $m=3$ and $\bG_{W_\lambda}^{\der}=\SO_3$,
\end{enumerate}
then there is a totally real extension $F'/F$ such that $\sigma_\lambda|_{\Gal_{F'}}$
is automorphic (i.e., attached to a regular algebraic, (polarized) cuspidal automorphic representation of $\GL_m(\A_{F'})$). 
Moreover, $\sigma_\lambda$ is part of a strictly (pure) compatible system of $F$.
\end{prop}

\begin{proof}
It suffices to apply Theorem \ref{pauto} on $\sigma_\lambda$ for almost all $\lambda$.
By the facts that $\{\rho_\lambda\}_\lambda$ is strictly compatible and $\sigma_\lambda$ has distinct $\tau$-Hodge-Tate numbers,
conditions \ref{pauto}(1),(3) (resp. conditions \ref{odd}(1),(3),(4)) hold for $\sigma_\lambda$ (resp. when $m=2$) 
for almost all $\lambda$ (see $\mathsection2.2$).
Since $\bG_{W_\lambda}^{\der}$ (of type A) is irreducible on $W_\lambda$,
Proposition \ref{use1} and Theorem \ref{general}(iv) imply that 
condition  \ref{pauto}(4) (resp. \ref{odd}(2)) holds for $\sigma_\lambda$ (resp. when $m=2$) for almost all $\lambda$.

Therefore, it remains to check condition \ref{pauto}(2) (odd essentially self-dual) 
in each case for almost all $\lambda$, i.e., the image of $\sigma_\lambda$ is contained in 
some $\GSp_m$ (resp. $\GO_m$) with totally odd (resp. totally even) similitude character.
Case (a) follows from Proposition \ref{odd}. Case (b) follows from the facts that the normalizer 
of $\SO_3$ in $\GL_3$ is $\GO_3$ and the similitude character is totally even \cite[Lemma 2.1]{CG13}.
\end{proof}

\section{Proofs of the results}
By Proposition \ref{use}, we may assume $\{\rho_\lambda\}_\lambda$ is also a 
Serre compatible system defined over $E$, i.e., the image of $\rho_\lambda$ is contained in $\GL_n(E_\lambda)$.

\subsection{Proof of Theorem \ref{mt1}}
It is trivial when $n=1$. We consider $n=2,3,4$ separately.\\

\textbf{(i) and (ii).} 
When $n=2$, there are two possibilities for the irreducible $\rho_{\lambda_0}$. 
If $\rho_{\lambda_0}$ is Lie-irreducible, then $\bG_{\lambda_0}^{\der}=\SL_2$. 
If $\rho_{\lambda_0}$ is not Lie-irreducible, then the distinct $\tau$-Hodge-Tate numbers 
and Clifford's theorem \cite{Cl37} imply that
$\rho_{\lambda_0}$ is induced from a character, i.e., $\rho_{\lambda_0}=\mathrm{Ind}_K^F \chi_{\lambda_0}$ where $[K:F]=2$.
In the first case, Theorem \ref{hui1} implies that $\bG_\lambda^{\der}=\SL_2$ and $\rho_\lambda$ is irreducible for all $\lambda$.
In the second case, since $\rho_{\lambda_0}$ and thus $\chi_{\lambda_0}$ are Hodge-Tate at all places above $\ell$ (condition \ref{scs}(ii)), 
$\chi_{\lambda_0}$ and thus $\rho_{\lambda_0}=\mathrm{Ind}_K^F \chi_{\lambda_0}$ can be extended to 
Serre compatible systems $\{\chi_{\lambda}\}_\lambda$ and $\{\mathrm{Ind}_K^F \chi_{\lambda}\}_\lambda$ (see e.g., \cite[Chapter III]{Se98}) after enlarging $E$.
Since the $\tau$-Hodge-Tate numbers are distinct for all $\lambda$ (condition \ref{scs}(iii)), 
$\rho_\lambda=\mathrm{Ind}_K^F \chi_{\lambda}$ is irreducible for all $\lambda$ by 
Mackey's irreducibility criterion \cite[$\mathsection7.4$]{Se77}.

When $n=3$, $\rho_{\lambda_0}$ is induced from a character if it is not Lie-irreducible and the treatment is identical to $n=2$ by using the regularity condition 
and Mackey's irreducibility criterion.
Hence, it suffices to consider Lie-irreducible $\rho_{\lambda_0}$.
There are two cases, either $\bG_{\lambda_0}^{\der}=\SL_3$ (of rank $2$) or 
$\bG_{\lambda_0}^{\der}=\SO_3$ (of rank $1$).
In the first case, we have $\bG_{\lambda}^{\der}=\SL_3$ (the maximal connected semisimple subgroup of $\GL_3$) for all $\lambda$ because 
the rank of $\bG_{\lambda}^{\der}$ is independent of $\lambda$ by Theorem \ref{hui1}.
In the second case, suppose $\rho_\lambda$ is reducible for some $\lambda$.
Since the formal character of $\bG_\lambda^{\der}$ is equal to those of $\bG_{\lambda_0}^{\der}=\SO_3$ (Theorem \ref{hui1})
which is 
\begin{equation}
\begin{pmatrix}
      x & 0 & 0 \\
       0& 1 &0 \\
			0&0& 1/x
\end{pmatrix},
\end{equation}
we obtain an irreducible decomposition
\begin{equation}\label{decom1}
\rho_\lambda=W_\lambda\oplus W_\lambda'
\end{equation}
such that  $(\dim W_\lambda,\dim W_\lambda')=(2,1)$ and 
 $\bG_{W_\lambda}^{\der}=\SL_2$ if $\bG_{W_\lambda}$ denotes the algebraic monodromy of $W_\lambda$.
If there are infinitely many $\lambda$ such that $\rho_\lambda$ is reducible, 
then Proposition \ref{use2} (case (a))
implies that $W_\lambda$ is part of a two dimensional strictly compatible system for some $\lambda$.
This, together with the fact that the character $W_\lambda'$ is part of a compatible system,
contradict that $\rho_{\lambda_0}$ is irreducible. 
It follows that $\rho_\lambda$ is irreducible and $\bG_\lambda^{\der}=\SO_3$ for almost all $\lambda$.
Hence, Proposition \ref{use2} (case (b)) implies that 
for some $\lambda$ and some totally real extension $F'/F$, $\rho_\lambda|_{\Gal_{F'}}$ is 
attached to a regular algebraic, polarized, cuspidal automorphic representation of $\GL_3(\A_{F'})$.
It follows that $\rho_\lambda(\Gal_{F'})\subset\GO_3$ for all $\lambda$ \cite[$\mathsection2$]{BLGGT14}.
We conclude that for all $\lambda$, $\bG_\lambda^{\der}=\SO_3$ (Theorem \ref{hui1}) and thus $\rho_\lambda$ is irreducible.

Finally, the $\lambda$-independence of $\bG_\lambda^{\der}$ established above together with Theorem \ref{hui1} imply that the identity component $\bG_\lambda^\circ$ is independent of $\lambda$.
\qed\\

\textbf{(iii).} Since $n\leq 3$, the representation $\rho_\lambda$ is of type A for every $\lambda$.
The assertion follows directly from  Proposition \ref{use1}, Theorem \ref{mt1}(i),
and Theorem \ref{general}(v) since $\{\rho_\lambda\}_\lambda$
is a strictly (and Serre) compatible system.\qed\\

\textbf{(iv).} When $n=4$, if $\rho_{\lambda_0}$ is not Lie-irreducible then again it must be induced by regularity (see \cite[Proposition 3.4.1]{Pat19}). 
Then if it is induced from a character, the treatment is identical to the above.
Otherwise,  $\rho_{\lambda_0}=\mathrm{Ind}_K^F \phi_{\lambda_0}$ where $[K:F]=2$ and 
$\phi_{\lambda_0}$ is two dimensional (not induced from a character)
and moreover if $\bG_{\phi_{\lambda_0}}$ is the algebraic monodromy of $\phi_{\lambda_0}$, 
the distinct $\tau$-Hodge-Tate numbers condition forces
the formal character of $\bG_{\phi_{\lambda_0}}^{\der}$ to be 
\begin{equation}\label{2fc}
\begin{pmatrix}
      x & 0  \\
			0& 1/x
\end{pmatrix}.
\end{equation}
If $\rho_\lambda$ is reducible for some $\lambda$, 
Theorem \ref{hui1} and \eqref{2fc}  imply that
the irreducible decomposition of such $\rho_\lambda$ is
\begin{equation}\label{decom2}
\rho_\lambda=W_\lambda\oplus W_\lambda'
\end{equation}
such that  $(\dim W_\lambda,\dim W_\lambda')=(2,2)$ and 
 $\bG_{W_\lambda}^{\der}=\SL_2=\bG_{W_\lambda'}^{\der}$ if 
$\bG_{W_\lambda}$ (resp. $\bG_{W_\lambda'}$) denotes the algebraic monodromy of $W_\lambda$ (resp. $W_\lambda'$).
By twisting $\{\rho_\lambda\}_\lambda$ with a big power of 
the system $\{\epsilon_\ell\}_\ell$ of $\ell$-adic cyclotomic characters,
we may assume $\bG_{W_\lambda}=\GL_2$.
As $\rho_{\lambda_0}=\mathrm{Ind}_K^F \phi_{\lambda_0}$, 
we obtain 
$\rho_{\lambda_0}\cong\rho_{\lambda_0}\otimes\eta$
where $\eta:\Gal_F\to\Gal(K/F)\to \{\pm 1\}$ is the non-trivial character.
By compatibility, we also obtain 
$$\rho_{\lambda}\cong\rho_{\lambda}\otimes\eta$$
for all $\lambda$. Since $\bG_{W_\lambda}=\GL_2$, the representations
$W_\lambda$ and $W_\lambda\otimes\eta$ are not isomorphic and
we deduce that $W_\lambda\cong W'_\lambda\otimes\eta$.
It follows that $W_\lambda|_{\Gal_K}\cong W_\lambda'|_{\Gal_K}$ but this contradicts the distinct $\tau$-Hodge-Tate numbers condition.
We conclude that $\rho_\lambda$ is irreducible for all $\lambda$.

It remains to consider $n=4$ and that $\rho_{\lambda_0}$ is Lie-irreducible.
We list the four cases: 
\begin{center}
\begin{tabular}{|l|l|c|c|} \hline
Case  & $(\bG_{\lambda_0}^{\der}, \rho_{\lambda_0})$ & Rank & Formal character\\ \hline
$(a)$ & $(\SL_2,\mathrm{Sym}^3(\mathrm{std}))$ 
& $1$ & $\tiny{\begin{pmatrix}
      x^2 & 0 &0 &0 \\
			0& x& 0& 0\\
			0&0 & 1/x&0 \\
		0	&0&0& 1/x^2
\end{pmatrix}}$ \\ \hline
$(b)$ &  $(\SO_4,\mathrm{std})$& $2$ &
$\tiny{\begin{pmatrix}
      x & 0 & 0& 0\\
			0& 1/x &0 & 0\\
			0&0 & y&0 \\
			0&0&0& 1/y
\end{pmatrix}}$ \\ \hline
$(c)$  &  $(\Sp_4,\mathrm{std})$  
& $2$ & $\tiny{\begin{pmatrix}
      x &  0& 0&0 \\
		0	& 1/x &0 & 0\\
			0& 0& y& 0\\
			0&0&0&1/y
\end{pmatrix}}$  \\ \hline
$(d)$ & $(\SL_4,\mathrm{std})$
& $3$ & $\tiny{\begin{pmatrix}
      x &0  & 0&0 \\
		0	& y&0 &0 \\
		0	& 0& z& 0\\
			0&0&0&1/(xyz)
\end{pmatrix}}$ \\ \hline
\end{tabular}
\end{center}
\vspace{.1in}
where ``$\text{std}$'' stands for standard representation.
For cases (a) and (d), the irreducibility of $\rho_\lambda$ for all $\lambda$ follows easily from 
the formal character of $\bG_{\lambda_0}^{\der}$ and Theorem \ref{hui1}.
For cases (b) and (c), the formal characters of $\bG_{\lambda_0}^{\der}$ are identical (given on the list).
If $\rho_\lambda$ is reducible for some $\lambda$, 
Theorem \ref{hui1} and the formal character imply that the only possible
 irreducible decomposition of $\rho_\lambda$ is
\begin{equation}\label{decom3}
\rho_\lambda=W_\lambda\oplus W_\lambda'
\end{equation}
such that  $(\dim W_\lambda,\dim W_\lambda')=(2,2)$ and 
 $\bG_{W_\lambda}^{\der}=\SL_2=\bG_{W_\lambda'}^{\der}$ if 
$\bG_{W_\lambda}$ (resp. $\bG_{W_\lambda'}$) denotes the algebraic monodromy of $W_\lambda$ (resp. $W_\lambda'$).
If there are infinitely many $\lambda$ such that $\rho_\lambda$ is reducible, 
then Proposition \ref{use2} (case (a))
implies that 
$W_\lambda$ (resp. $W_\lambda'$) is part of a two dimensional strictly compatible system for some $\lambda$.
This contradicts that $\rho_{\lambda_0}$ is irreducible. 
It follows that $\rho_\lambda$ is irreducible for almost all $\lambda$.
\qed

\subsection{Proof of Theorem \ref{mt2}}
By twisting $\{\rho_\lambda\}_\lambda$ with a big power of 
the system of $\ell$-adic cyclotomic characters, we assume  $\bG_\lambda$ is connected
for all $\lambda$ \cite{Se81} and $\bG_{\lambda_0}=\GSp_4$.\\

\textbf{(i).} Suppose $\rho_{\lambda_0}$ is fully symplectic for some $\lambda_0$.  
Theorem \ref{mt1}(iv) and the information on formal character (Theorem \ref{hui1}) 
imply that for almost all $\lambda$,
the algebraic monodromy $\bG_\lambda$ is either $\GSp_4$ or $\GO_4^\circ$.
In the $\GO_4^\circ$-case, the representation $\rho_\lambda\otimes\rho_\lambda^\vee$ 
has a three dimensional irreducible factor $W_\lambda$ with algebraic monodromy $\SO_3$ because 
of the short exact sequence
\begin{equation}\label{ses}
1\to \GL_1\to \GL_2\times\GL_2\to \GO_4^\circ\to 1
\end{equation}
given by the exterior tensor of the standard representation of $\GL_2$ with itself.
The factor $W_\lambda$ has distinct Hodge-Tate numbers.
Indeed, if $\ell$ is the residue characteristic of $\lambda$, \cite[Corollary 3.2.12]{Pat19} asserts that
there is a Hodge-Tate lift of $\rho_\lambda:\Gal_{\Q_\ell}\to \GO_4^\circ(\overline{E}_\lambda)$ to 
$$f_\lambda\oplus f_\lambda': \Gal_{\Q_\ell}\to\GL_2(\overline{E}_\lambda)\times\GL_2(\overline{E}_\lambda).$$
Since $\rho_\lambda|_{\Gal_{\Q_\ell}}=f_\lambda\otimes f_\lambda'$ has distinct Hodge-Tate numbers, 
both the two dimensional $f_\lambda$ and $f_\lambda'$
have distinct Hodge-Tate weights. Hence, $W_\lambda|_{\Gal_{\Q_\ell}}=\text{ad}^0 f_\lambda$ (or $\text{ad}^0 f_\lambda'$)
also has distinct Hodge-Tate weights. 

If there is an infinite set $\mathcal{L}$ of $\lambda$ such that $\bG_\lambda=\GO_4^\circ$,
then Proposition \ref{use2} (case (b) applied to the system $\{\rho_\lambda\otimes\rho_\lambda^\vee\}_\lambda$)
 asserts that for some $\lambda_1\in\mathcal{L}$, the three dimensional $W_{\lambda_1}$ is part of a strictly compatible system $\{\phi_\lambda\}_\lambda$.
The algebraic monodromy of $\phi_\lambda$ is equal to $\SO_3$ for all $\lambda$ 
by Theorem \ref{mt1}(ii).
Consider the compatible system 
$$\{\phi_\lambda\oplus(\rho_\lambda\otimes\rho_\lambda^\vee)\}_\lambda$$
 and let $\bH_\lambda$ be the algebraic monodromy at $\lambda$.
By construction, the semisimple ranks of both $\bH_{\lambda_1}$ and $\bG_{\lambda_1}$ are equal to $2$.
By $\bG_{\lambda_0}=\GSp_4$ and Goursat's lemma, 
the semisimple rank of $\bH_{\lambda_0}$ is equal to $3$ which is 
greater than $2$, the semisimple rank of $\bH_{\lambda_1}$.
This contradicts Theorem \ref{hui1}. \qed\\

\textbf{(ii).} 
 Theorem \ref{mt2}(i) implies that  
there is a Hodge-Tate character $\mu_\lambda$ such that
\begin{equation}\label{equal}
\rho_\lambda\cong\rho_\lambda^\vee\otimes \mu_\lambda
\end{equation} 
for almost all $\lambda$.
When $\bG_\lambda=\GSp_4$, the trace of $\rho_\lambda(Frob_p)\neq0$ for a Dirichlet density one set of rational primes $p$.
Hence by \eqref{equal} and the compatibility (\ref{scs}(a)) of $\{\rho_\lambda\}_\lambda$ (resp. $\{\rho_\lambda\}_\lambda^\vee$),
for almost all $\lambda$ the similitude characters $\mu_\lambda$ and $\mu_{\lambda_0}$ are compatible,
i.e., for almost all rational primes $p$ both characters are unramified and $\mu_\lambda(Frob_p)=\mu_{\lambda_0}(Frob_p)\in E$.
Since  $\mu_{\lambda_0}$ is odd, $\mu_\lambda$ is also odd for almost all $\lambda$. 
We break down the proof into three steps:
\begin{enumerate}
\item[(ii-1)] prove that $\rho_\lambda$ is residually irreducible for almost all $\lambda$;
\item[(ii-2)] prove that $\{\rho_\lambda\}_\lambda$ is potentially automorphic;
\item[(ii-3)] prove that $\bar\rho_\lambda(\Gal_\Q)$ has a subgroup conjugate to $\Sp_4(\F_\ell)$
for almost all $\lambda$.\\
\end{enumerate}

(ii-1). By Proposition \ref{use1}, there exist algebraic envelopes 
$\uG_\lambda$ (Theorem \ref{general}(i)) of $\rho_\lambda$ for almost all $\lambda$.
Then Theorem \ref{general}(iii) and the construction of $\uG_\lambda$ (\cite[Proposition 3.9(iii)]{Hu22})
imply that $\uG_\lambda$ is either $\GSp_4$ or the first group in
\begin{equation}\label{redu}
\mathbb{G}_m(\SL_2\times\SL_2)\subset\GL_2\times\GL_2\subset\GL_4\hspace{.2in}\text{(reducible action)},
\end{equation} 
where $\mathbb{G}_m$ denotes the set of scalars in $\GL_4$.
 If $\rho_\lambda$ is residually reducible 
then $\uG_\lambda\neq \GSp_4$ by Theorem \ref{general}(ii).
Thus, $\uG_\lambda$ has to be $\mathbb{G}_m(\SL_2\times\SL_2)$ in \eqref{redu} and the semisimplified reduction $\bar\rho_\lambda^{\ss}$
decomposes into a sum of two two-dimensional absolutely irreducible representations corresponding to \eqref{redu},
\begin{equation}\label{split}
\bar\rho_\lambda^{\ss}=\overline W_\lambda\oplus\overline {W'}_\lambda.
\end{equation}  
By $\bar\rho_\lambda^{\ss}\cong \bar\rho_\lambda^{\vee,\ss}\otimes\bar\mu_\lambda$ (semisimplified reduction of \eqref{equal}) 
and Theorem \ref{general}(ii),
we obtain 
\begin{equation}\label{polar}
\overline W_\lambda\cong \overline W_\lambda^\vee\otimes \bar\mu_\lambda\hspace{.2in}\text{and}\hspace{.2in}
\overline {W'}_\lambda\cong \overline {W'}_\lambda^\vee\otimes \bar\mu_\lambda.
\end{equation}
Since $\overline W_\lambda$ is not induced, 
it follows that $\text{det}\overline W_\lambda= \bar\mu_\lambda= \text{det}\overline{W'}_\lambda$ is odd.

Let $\epsilon_\ell$ (resp. $\bar\epsilon_\ell$) be the $\ell$-adic (resp. mod $\ell$) cyclotomic character.
Suppose there are infinitely many $\lambda$ such that $\bar\rho_\lambda^{\ss}$ is not absolutely irreducible (i.e., \eqref{split} holds).
Then the two-dimensional Galois representation $\overline W_\lambda$ of $\Q$ is odd irreducible.
We now follow the arguments in \cite[$\mathsection4.5.1$]{Hu22} (also \cite{Die02a,Die02b}) 
using Serre's modularity conjecture \cite{Se87}. By applying (strong form) Serre's modularity conjecture (proved in \cite{KW09a,KW09b}),
there exist an integer $m$ and a cuspidal Hecke eigenform $f$ such that if
\begin{equation}\{\psi_{\lambda,f}:\Gal_\Q\to \GL_2(E_\lambda)\}_\lambda
\end{equation}
is the Serre compatible system attached to $f$ (by enlarging $E$ if necessary),
then the semisimplified reduction satisfies
\begin{equation}
\bar\psi_{\lambda,f}^{\ss}\cong \overline W_\lambda\otimes\bar\epsilon_\ell^m
\end{equation}
for infinitely many $\lambda$\footnote{Let $f_\lambda$ be the modular form attached to $\overline W_\lambda$ by Serre's recipe (see \cite[$\mathsection2$]{Da95}). The point is that, the crystallinity (by \ref{scs}(b)) and $\lambda$-independence of $\tau$-Hodge-Tate numbers (by \ref{scs}(c)) give control on the weight of $f_\lambda$ and the strong local-global compatibility (by \ref{scs}(a,d)) gives control on the level of $f_\lambda$. These data allow us to find an integer $m$ and a modular form $f$ attached to 
$\overline W_\lambda\otimes\epsilon_\ell^m$ for infinitely many $\lambda$.}. 
Since $[\Gal_\Q,\Gal_\Q]$ is (absolutely) irreducible on $\overline W_\lambda$
by \eqref{redu} and Theorem \ref{general}(ii), the algebraic monodromy group of $\psi_{\lambda,f}$ contains $\SL_2$ (for all $\lambda$).
 Consider the $6$-dimensional Serre compatible system (satisfying the conditions (a) and (b) 
of Theorem \ref{general})
given by direct sum:
\begin{equation}\label{sum}
\{U_\lambda:=(\rho_{\lambda}\otimes\epsilon_\ell^m)\oplus\psi_{\lambda,f}\}_\lambda.
\end{equation}
Let $\bM_\lambda$ be the algebraic monodromy group (resp. $\uM_\lambda$ be the 
algebraic envelope) of $U_\lambda$ at $\lambda$ (resp. for almost all $\lambda$). 
Theorem \ref{general}(iii) implies that 
the semisimple groups $\bM_{\lambda}^{\der}$ and 
$\uM_{\lambda}^{\der}$
have the same formal bi-character on respectively $U_\lambda$ and
\begin{equation}\label{sum}
\overline U_\lambda^{\ss}:=(\bar{\rho}_{\lambda}^{\ss}\otimes\bar\epsilon_\ell^m)\oplus\bar{\psi}_{\lambda,f}^{\ss}
\end{equation}
for almost all $\lambda$.
This is impossible since for infinitely many $\lambda$
we have $\bM_{\lambda}^{\der}=\Sp_4\times\SL_2$ (by Goursat's lemma)
and $\uM_{\lambda}^{\der}\cong\uG_\lambda^{\der}=\SL_2\times\SL_2$ 
(as the second factor of \eqref{sum} is a subrepresentation of the first factor)
are of different ranks. We conclude that $\rho_\lambda$ is residually irreducible
for almost all $\lambda$. \\

(ii-2). Recall that the algebraic envelope $\uG_\lambda$ is either $\GSp_4$ or $\mathbb{G}_m(\SL_2\times\SL_2)$ in \eqref{redu} for almost all $\lambda$. 
If $\uG_\lambda$ is $\mathbb{G}_m(\SL_2\times\SL_2)$ in \eqref{redu} 
for infinitely many $\lambda$, then except finitely many such $\lambda$ the
restriction $\bar\rho_\lambda|_{\Gal_{L}}$ is a sum of two non-isomorphic two dimensional irreducible representations for some finite Galois extension $L/\Q$ 
by Theorem \ref{general}(i) and (ii). Hence, it follows by \cite{Cl37} that 
$\bar\rho_\lambda$ is induced from a two dimensional representation of a 
field $Q_\lambda$ such that $[Q_\lambda:\Q]=2$ and $Q_\lambda\subset L$\footnote{The field $Q_\lambda$ corresponds to the index two subgroup 
$\bar\rho_\lambda(\Gal_\Q)\cap (\GL_2\times\GL_2)$.}
 for infinitely many $\lambda$. Hence, we can fix $Q_\lambda$ equal to
a  degree two extension $Q$
of $\Q$ for infinitely many $\lambda$. This implies
the existence of a half Dirichlet density set of rational primes $p$ 
such that $\rho_\lambda(Frob_p)$ has zero trace. Since this contradicts
$\bG_{\lambda_0}=\GSp_4$ by the compatibility condition \ref{scs}(a), this case is impossible.

If $\uG_\lambda$ is $\GSp_4$ for almost all $\lambda$, then 
the commutants of $\Gal_{\Q(\zeta_\ell)}$ and $\Sp_4$ in $\End(\overline \F_\ell)$
are equal for almost all $\lambda$ by Theorem \ref{general}(ii). 
This implies that the condition \ref{pauto}(4) holds 
(and thus $\bar\rho_\lambda^{\ss}$ is absolutely irreducible) for almost all $\lambda$. Since the conditions \ref{pauto}(1)--(3)
also hold for almost all $\rho_\lambda$ by the strict compatibility
of the system, Theorem \ref{mt2}(i), and the oddness of similitude character $\mu_\lambda$,
we obtain by Theorem \ref{pauto} that
the representation $\rho_\lambda$ (and thus the system $\{\rho_{\lambda}\}_\lambda$) is potentially automorphic.\\

(ii-3). It remains to prove that $\bar\rho_\lambda(\Gal_\Q)$ contains a group isomorphic to $\Sp_4(\F_\ell)$ for almost all $\lambda$.
Since the algebraic envelope $\uG_\lambda$ is $\GSp_4$ for almost all $\lambda$
and $\uG_\lambda$ is the image in the $\lambda$-component of the algebraic envelope $\uG_\ell\subset\GL_{4[E:\Q],\F_\ell}$ 
of 
$$\rho_\ell:=\bigoplus_{\lambda|\ell}\rho_\lambda:\Gal_\Q\to\prod_{\lambda|\ell}\GL_4(E_\lambda)
=(\mathrm{Res}_{E/\Q}\GL_4)(\Q_\ell)\subset\GL_{4[E:\Q]}(\Q_\ell)$$
given by the restriction of scalars \cite[$\mathsection3.4$]{Hu22}, it follows that
 the universal cover $\uG_\ell^{\sc}$ of the derived group of $\uG_\ell$ satisfies\footnote{The isomorphism \eqref{isomsp} follows from \cite[$\mathsection1.3$]{Kn67}
and the fact that $\Sp_{4,\F_q}$ has only one $\F_q$-form up to isomorphism.}
\begin{equation}\label{isomsp}
\uG_\ell^{\sc}\cong \prod_i \mathrm{Res}_{\F_{q_i}/\F_\ell}\Sp_{4,\F_{q_i}}
\end{equation}
for all sufficiently large $\ell$, where $q_i$ is a power of $\ell$. Since for $\ell\gg0$ the index 
$$[\uG_\ell^{\der}(\F_\ell):\bar\rho_\ell^{\ss}(\Gal_\Q)\cap \uG_\ell^{\der}(\F_\ell)]\leq C$$
for some $C>0$ independent of $\ell$ by \cite[Theorem 2.11(ii)]{Hu22}
and $\uG_\ell^{\sc}(\F_\ell)$ is generated by elements of order $\ell$ \cite[Theorem 12.4]{St68},
it follows that
$\bar\rho_\ell^{\ss}(\Gal_\Q)$
contains the image of $\uG_\ell^{\sc}(\F_\ell)$ in $\uG_\ell^{\der}(\F_\ell)$ for $\ell\gg0$.
Therefore, \eqref{isomsp} and the fact that the representation
$$\prod_j\Sp_{4,\overline\F_\ell}\cong\uG_\ell^{\sc}\times\overline\F_\ell\to \uG_\ell^{\der}\times\overline\F_\ell\twoheadrightarrow 
\uG_\lambda^{\der}=\Sp_{4,\overline\F_\ell}\subset\GL_{4,\overline\F_\ell}$$
is given by projection (as very automorphism of $\Sp_{4,\overline\F_\ell}$ is inner) to one of the $\Sp_4$-factors 
imply that the image of the absolutely irreducible $\bar\rho_\lambda$ 
 contains a subgroup conjugate to $\Sp_4(\F_{\ell})$
for almost all $\lambda$.
\qed

\subsection{Proof of Corollary \ref{mc}}
The proof we give here is similar to \cite[Theorem 8]{HL16}.
Suppose $\ell$ is sufficiently large.
By Theorem \ref{mt2}(i), we assume 
the image $\Gamma_\ell$ (resp. $[\Gamma_\ell,\Gamma_\ell]$) is 
a subgroup of $\GSp_4(\Q_\ell)$ (resp. $\Sp_4(\Q_\ell)$).
Let $\Delta_\ell\subset\Sp_4(\Q_\ell)$ be a maximal compact subgroup
containing $[\Gamma_\ell,\Gamma_\ell]$. There is an affine smooth group scheme $\cH_\ell/\Z_\ell$
such that the generic fiber $\cH_\ell\times\Q_\ell=\Sp_{4,\Q_\ell}$ and $\cH_\ell(\Z_\ell)=\Delta_\ell$ by Bruhat-Tits theory \cite[3.4.3]{Ti79}. Since $\Sp_{4,\Q_\ell}$ is simply-connected,
the special fiber $\cH_\ell\times\F_\ell$ is connected \cite[3.5.3]{Ti79}.
By $[\Gamma_\ell,\Gamma_\ell]\subset\Delta_\ell$, the big image part of
Theorem \ref{mt2}(ii), and the facts that $\Sp_4(\F_\ell)$ is perfect for $\ell>2$
and the kernel of the reduction map $r_\ell:\cH_\ell(\Z_\ell)\to\cH_\ell(\F_\ell)$ is pro-$\ell$, 
the composition series of the image $r_\ell([\Gamma_\ell,\Gamma_\ell])$ contains 
the finite simple group of Lie type
\begin{equation}\label{2gp}
\Sp_4(\F_\ell)/\{\pm 1\}.
\end{equation}

Since the \emph{$\ell$-dimension} of $\Sp_4(\F_\ell)/\{\pm 1\}$ is 
$\dim\Sp_{4,\F_\ell}$ (see \cite[$\mathsection2$]{HL16} for definition)
which is equal to $\dim(\cH_\ell\times\F_\ell)$ by the smoothness of $\cH_\ell/\Z_\ell$, 
it follows from \cite[Corollary 6(iv)]{HL16} that $\cH_\ell\times\F_\ell$ is semisimple.
Hence, $\cH_\ell$ is a semisimple group scheme over $\Z_\ell$.
Moreover, the (simply-connected) special fiber is $\Sp_{4,\F_\ell}$ by footnote $9$ 
and $\Delta_\ell=\cH_\ell(\Z_\ell)$ is a hyperspecial maximal compact subgroup 
of $\Sp_4(\Q_\ell)$ (after \cite[Remark 7.2.13]{Co14}) 
isomorphic to $\Sp_4(\Z_\ell)$ \cite[2.5]{Ti79}.
Since \eqref{2gp} is a composition factor of $r_\ell([\Gamma_\ell,\Gamma_\ell])$
and $\Sp_4(\F_\ell)$ does not have an index two subgroup, 
we obtain
$$r_\ell([\Gamma_\ell,\Gamma_\ell])=\Sp_4(\F_\ell)=\cH_\ell(\F_\ell).$$
 It follows from \cite[Theorem 1.3]{Va03} that $[\Gamma_\ell,\Gamma_\ell]=\cH_\ell(\Z_\ell)=\Delta_\ell$ 
for $\ell\gg0$. We are done 
because the compact subgroup $\Gamma_\ell^{\sc}\subset\Sp_4(\Q_\ell)$ contains
the maximal compact subgroup $[\Gamma_\ell,\Gamma_\ell]=\Delta_\ell\cong\Sp_4(\Z_\ell)$.
\qed

\vspace{.1in}
\end{document}